\documentclass[12pt]{amsart}

\usepackage{amsfonts,ifthen, amssymb,indentfirst}
\usepackage[all]{xypic}

\newtheorem{prop}{Proposition}[section]
\newtheorem{theorem}[prop]{Theorem}
\newtheorem{lemma}[prop]{Lemma}

\newtheorem{defi}[prop]{Definition}
\newtheorem{coro}[prop]{Corollary}
\newtheorem{remark}[prop]{Remark}

\newcommand{\cqd}{\hfill$\Box$}

\newcommand{\cdbc}[1]{{D}^b_c(#1)}

\address{ Universitat, Departament de Matematiques, Edifici C, Facultad de Ciencies, 08193 Bellaterra, Barcelona}
\author[Cristina Mart{\'\i}nez]{Cristina Mart{\'\i}nez}

\email{cmartine@mat.uab.cat}

\address{University of California at Santa Cruz, Department of Mathematics, Santa Cruz, CA 95064, Chinese University of Hong Kong, Institute of Mathematics, Shatin, Hong Kong}
\author[Andrey Todorov]{ Andrey Todorov}
\email{andreytodorov@msn.com}

\begin{document}

\title{ DERIVED EQUIVALENCES OF CALABI-YAU FIBRATIONS} 

\date{\today}


\begin{abstract}
 We consider  fibrations by abelian surfaces and $K3$ surfaces over a one dimensional base that are Calabi-Yau and we obtain  dual fibrations that are derived equivalent to the original fibration. Finally, we relate the problem to mirror symmetry.

\end{abstract}
\maketitle
{\small \tableofcontents }

\section{Introduction}

One of the main motivations to the study of algebraic varieties via
their categories of coherent sheaves is its relation with string
theory. The conformal field theory associated to a variety in string
theory contains a lot of information packaged in a
categorical way rather than in directly geometric terms. 
A second motivation to study varieties via their sheaves is that
this approach is expected to generalize more easily to
non-commutative varieties and categorical methods enable one to
obtain a truer description of certain varieties than current
geometric techniques allow. For example many equivalences relating
the derived categories of pairs of varieties are now known to exist.

Let $X$ be a smooth complex projective variety over a field $k$ with
structure sheaf $\mathcal{O}_{X}$ and $D^{b}(X)$ the bounded derived
category of coherent sheaves. It is an interesting question how much
information about $X$ is contained in $D^{b}(X)$. Certain invariants
of $X$ can be shown to depend only on $D^{b}(X)$. Because of the
uniqueness of Serre functors, the dimension of $X$ and whether it is
Calabi-Yau or not, can be read off from $D^{b}(X)$.


Let $X, Y$ be two smooth complex projective varieties. The
Fourier-Mukai tranform $FMT$ by an object $\mathcal{P}\in
D^{b}(X\times Y)$ is the exact functor:
$$\phi_{\mathcal{P}}:D^{b}(X) \rightarrow D^{b}(Y)$$
$$\phi_{\mathcal{P}}(\mathcal{G})=R\pi_{2*}(\pi_{1}^{*}\mathcal{G}\otimes^{L} \mathcal{P}),$$
where $\pi_{1}$ and $\pi_{2}$ are the projections maps over the
first and second component respectively.

If $\phi_{\mathcal{P}}$ is an equivalence, then $\mathcal{P}$ must
satisfy a partial Calabi-Yau condition:

\begin{equation}\label{CYcond}
\mathcal{P}\otimes \pi_{1}^{*}K_{X}\otimes \pi_{2}^{*}
K_{Y}^{-1}\cong \mathcal{P}.
\end{equation}


From a topological point of view, complexes of coherent sheaves are
D-branes of the B-model (twist of a N = 2 SCFT), and morphisms
between the objects of $D^{b} (X )$ are identified with the states
of the topological string and composition of morphisms is computed
by correlators of the B-model. Originally, a $D-$brane in string
theory is by definition an embedding of a manifold, variety or cycle
in the space-time that serves as a boundary condition for
open-strings moving in the space-time.

Our main aim in this context is to obtain interpretations and geometric
 criteria for the equivalence between the derived categories of Calabi-Yau
manifolds, in particular in the case of Calabi-Yau $n$-folds that admits a fibration structure
by abelian varieties or $K3$ surfaces.

In the first two sections, we will study Calabi-Yau spaces that are fibered over the same
base $B$ by polarized abelian varieties. If $X/B$ is an abelian
fibration with a global polarization, and we call $X^{\vee}/B$ its
dual fibration, our main result is:

\begin{theorem} \label{T1}The derived categories of both fibrations $X/B$ and $X^{\vee}/B$ are equivalent and there is an equivalence
$\phi_{b}:D^{b}(X_{b})\rightarrow
D^{b}(\widehat{X}_{b})$ for every closed point $b\in B$.
\end{theorem}
If $p:X\rightarrow C$ is a $K3$ fibration over an algebraic curve $C$, we prove the following theorem:

\begin{theorem}\label{T2}
 Given a non singular fibration $p:X\rightarrow C$ by $K3$ surfaces  with a fixed polarization $l$,
 there exists at least a dual fibration which is derived equivalent to the original one and corresponds to a connected component of the relative
 moduli space $\mathcal{M}^{l}(X/C)$.  
\end{theorem}

\subsubsection*{Acknowledgments}
The author thank Cirget at UQAM (Montreal) and MPIM (Bonn) where
this project started for excelent working conditions.

\section{Abelian fibrations.}\label{abelianfib}

Let $\Gamma\cong \mathbb{Z}^{2d}$ be a lattice in a real vector
space $U$ of dimension $2d$ or complex vector space of dimension $d$,
and let $\Gamma^{*}\subset U^{*}$ be
the dual lattice. The complex torus $(U/\Gamma)$ is an abelian
variety $A$ of dimension $d$ over $\mathbb{Z}$ if it is algebraic,
that is, it has embedding to the projective space. The dual
abelian variety $\hat{A} $ is given by the dual torus $(U^{*}/\Gamma^{*})$. There is a unique line bundle $P$ on the product
$A\times \hat{A}$ such that for any point $\alpha \in \hat{A}$ the restriction $P_{\alpha}$ on $A\times \{\alpha\}$ represent an element
of $Pic^{0}(A)$ corresponding to $\alpha$, and the restriction $P|_{{\{0\}\times \hat{A}}}$ is trivial. Such $P$ is called the Poincare line bundle.

Let $\pi: X\rightarrow S$ be a proper map, when all the fibers are
equidi mensional, we say that it is a fibration. We don't impose
further assumptions on the base. It is of interest from the point of
view of Mirror Symmetry, the case in which the total fibration is
Calabi-Yau. In this case, Strominger, Yau and Zaslow (\cite{SYZ})
have conjectured how would it be the structure of the mirror
fibration. Mirror dual Calabi-Yau manifolds should be fibered over
the same base in such a way that generic fibers are dual tori, and
each fiber of any of these two fibrations is a Lagrangian
submanifold.

\begin{defi} We say that a fibered morphism $\pi: X\rightarrow S$ is a Calabi-Yau fibration,
if it is a connected morphism of non-singular projective varieties whose general fibre has trivial canonical bundle and
such that $K_{X}\cdot C=0$ for any curve $C$ contained in a fibre of $\pi$.
\end{defi}


Given a fibration of abelian varieties over the unit disc
$\Delta\subset \mathbb{C}=\{t \in\mathbb{C} |\, |t|<1\}$, such that, the
fiber consists of the product of an abelian variety with its dual
$A\times A^{\vee}$, there is a natural polarization defined, by
considering the product $\pi^{*}_{1} \mathcal{P}_{A}\otimes
\pi_{2}^{*}\mathcal{P}_{\widehat{A}}$, where $\mathcal{P}_{A}$ and
$\mathcal{P}_{\widehat{A}}$ are the respective Poincare bundles
over $A\times \widehat{A}$ and $\widehat{A}\times A$. Due to a
result of Deligne (see Theorem 1.11 of \cite{Del}), one can extend this polarization
to the all family. It is of special interest when the family has
the structure of a scheme, these are the abelian schemes that can be seen as schemes in groups.

\begin{defi} Let $S$ be a noetherian scheme. A group scheme $\pi: X\rightarrow
S$ is called an abelian scheme if $\pi$ is smooth and proper, it has
a global section and the geometric fibers of $\pi$ are connected.
\end{defi}

For abelian schemes Mukai has introduced the Fourier-Mukai
transform in the early 80's as a duality among sheaves on abelian
varieties. In general we will allow singular fibers.

Let us start with an abelian fibration $\pi:X\rightarrow B$ which is Calabi-Yau, that
is, a projective smooth morphism whose fibers are abelian
varieties admitting a global polarization, that is, a very ample
line bundle $\mathcal{L}$ on it, in particular this means that
there is a global embedding $(X/B)\hookrightarrow \mathbb{P}^{N}$.
We denote by $X_{b}$, the fiber of $\pi$ over $b\in B$.



We are interested in defining a dual fibration $X^{\vee} /B$ in such
a way that over the smooth locus, the fibers correspond to the dual
abelian varieties of the original fibration, that is, if
$\mathcal{P}$ is the Poincar\'e sheaf, then $\forall s \in S$,
$\mathcal{P}_{b}$ is the Poincar\'e bundle over $X_{b} \times
X^{\vee}_{b}$. The corresponding derived equivalence of the fibers
$X_{b}$ and $X^{\vee}_{b}$ is given by the Poincar\'e bundle over
the product $X_{b}\times X^{\vee}_{b}$.
Therefore the dual fibration  $\widehat{\pi}: X^{\vee}\rightarrow B$ admits also a global polarization.

 Let $\rho:X\rightarrow B$ be an abelian fibration, admitting
a global polarization and we don't make further assumptions on the
fibers. We consider the moduli problem of its dual fibration, that is, the dual fibration as the stack representing 
the moduli functor of semistable sheaves on the fibres that contains line bundles of degree 0 on smooth fibres. The corresponding coarse moduli space is not a fine moduli space due to the presence of singular fibres. 
If the fibre $Y_{b}$ is singular, the fibre $Y_{b}^{\vee}$ is not
isomorphic to $Y_{b}$. It can be seen that if $Y_b$ is a cycle of
projective lines, then $Y_b^{\vee}$ is the nodal rational cubic.
Moreover, there are no FM equivalences between the derived
categories of  $Y_{b}$ and $Y_{b}^{\vee}$ in this case. Since a
relative integral functor  $D(Y) \to D(Y^{\vee})$ is an equivalence
if and only if $D(Y_{b}) \to D(Y_{b}^{\vee})$ is an equivalence for
every fibre (see Prop. 2.15 of \cite{HLS}) one has that the
Poicar\'e sheaf on $Y\times_B Y^{\vee}$ does not induce an
equivalence of derived categories, unless when all the fibres are
integral. The stack representing the functor may be a FM partner for
the original fibration using the universal family as kernel for the
FM transform and we expect that is a gerbe over an algebraic space
that would correspond to a twisted version of $Y/B$. It will be
explained later in the next section how to twist a given fibration.


\begin{defi} We define the dual fibration $\widehat{\rho}: \widehat{X}\rightarrow B$ as the moduli stack representing the relative
Poincar\'e sheaf $\mathcal{E}$.
\end{defi}


\begin{theorem}\label{mainth}
The integral functor $\phi_{\mathcal{E}}^{X\rightarrow
\widehat{X}}: D^{b}(X)\rightarrow D^{b}(\widehat{X})$ is an
equivalence if and only if $\phi_{b}:D^{b}(X_{b})\rightarrow
D^{b}(\widehat{X}_{b})$ is an equivalence for every closed point
$b\in B$.
\end{theorem}

{\bf Proof.} First we fix our attention on the smooth locus. Let
$\Sigma(\pi) \hookrightarrow B$ be the discriminant locus of
$\pi$, that is, the closed subvariety in the parameter space $B$
corresponding to the singular fibers.

\begin{equation}\label{diagram}\begin{array}{cccc}

 X^{\vee}& \supset X^{\vee}- \pi^{-1}(\Sigma(\pi)) & \hookrightarrow & \mathbb{P}^{N} \\
 \downarrow &   & & \\
B &  \supset  B-\Sigma(\pi) & & \\
\end{array}\end{equation}

We can take then the Zariski closure of
$X^{\vee}-\pi^{-1}(\Sigma(\pi))$ in $\mathbb{P}^{N}$. For each $b
\in B-\Sigma(\pi)$, the corresponding derived equivalence of the
fibers $X_{b}$ and $X^{\vee}_{b}$ is given by the Poincare bundle
$\mathcal{P}_{b}$ over the product $X_{b}\times X^{\vee}_{b}$.
Moreover, its first Chern class $c_{1}(\mathcal{P}_{b})$ lives in
$H^{1,1}(X_{b}\times X^{\vee}_{b},\mathbb{Z})\cap
H^{2}(X_{b}\times X^{\vee}_{b},\mathbb{Z})$. The monodromy group is defined by the action of the fundamental group
of the complement of the discriminant locus $\pi_{1}(B-\Sigma(\pi))$ on the cohomology $H^{*}(X_{b}\times X^{\vee}_{b},\mathbb{Z})$
of a fixed non-singular fiber, 
 and since the class of the polarization is invariant by the monodromy, by Deligne
theorem we can extend the class of the Poincare bundle
 to the non singular fibers, it is the relative Poincare
sheaf of the fibred product of the two families over the base $B$
and we will call it $\mathcal{E}$.

Let $X^{sm}:=X|_{B-\Sigma\,(p)}$ be the fibration restricted to the smooth locus $\widetilde{B}:=B-\Sigma\,(p)$. The family has
the structure of a scheme, these are the abelian schemes that can be seen as schemes is groups and have already been studied by Mukai.
The Picard functor is representable by the dual fibration $\widehat{X}^{sm}$ and there is a global section $\sigma: \widetilde{B}\rightarrow X^{sm}$. The dual fibration $\widehat{X}^{sm}$ is defined in such a way that the fibres correspond to the dual abelian varieties of the original fibration, that is, if $\mathcal{E}$ is the relative Poincar\'e sheaf, then $\forall b\in B$, $\mathcal{P}_{b}$ is the Poincar\'e bundle over $X_{b}\times \hat{X}_{b}$. The fibres $X_{b}$ and $\hat{X}_{b}$ are derived equivalent, and the equivalence is given by the Poincar\'e bundle over the product $X_{b}\times \hat{X}_{b}$.
$X^{sm}$ and $\widehat{X^{sm}}$ can be identified by $x\rightarrow \mathbf{ m}_{x}\otimes \mathcal{O}_{X_{t}}(\sigma(t))$, where $\mathbf{ m}_{x}$ is the ideal sheaf of the point $x$ in the fibre $X_{t}, (t=p(x))$ and there is a natural polarization by considering the product $\pi_{1}^{*}\mathcal{O}_{X^{sm}}(\Theta)\otimes \pi_{2}^{*}\mathcal{O}_{\widehat{X}^{sm}}(\Theta)$, where $\Theta:=\sigma(\widetilde{B})$ and $\pi_{1}$, $\pi_{2}$ are the projection maps of $X^{sm}\times \widehat{X}^{sm}$ over the first and second components.
Then the Fourier-Mukai transform $\phi_{\mathcal{E}}: D^{b}(X^{sm})\rightarrow D^{b}(\widehat{X}^{sm})$ with kernel $\mathcal{E}$,
\noindent defines an equivalence of the corresponding derived categories. Thus the abelian schemes $X^{sm}$ and $\widehat{X}^{sm}$ are derived equivalent and the equivalence is given by the FMT with kernel the relative Poincar\'e sheaf, $$\phi_{\mathcal{E}}(L)=\mathbf{ R}\pi_{2*}(\pi_{1}^{*}L\otimes \mathcal{E}),$$
where $\pi_{1}$ and $\pi_{2}$ are the projection maps over the first and second components of the fibred product $X\times_{B}\widehat{X}$.


Now let us consider a Galois covering of the base, that is, an etale finite covering $\sigma: B'\rightarrow B$ that maps surjectively to
the base,  (in the case $B$ is a curve, locally around a point $y\in B'$ and $x=\sigma
(y)$, $\sigma$ is simply the function $\{z \in \mathbb{C}: |z|<1
\} \rightarrow \{z \in \mathbb{C}| |z|<1 \} $ given by
$z\rightarrow z^{k}$, where $k$ is the multiplicity of $\sigma$ at
$y$). We  assume that $\sigma: B'\rightarrow B$ is a
Galois covering with finite Galois group $G$, just we observe that
any normal extension of fields admits a Galois extension. If
$\mathcal{M}(B), \mathcal{M}(B')$ are the fields of meromorphic
functions on $B$ and $B'$ respectively, $\sigma^{*}:
\mathcal{M}(B)\rightarrow \mathcal{M}(B')$ is a Galois field
extension of degree $k$, with Galois group $G$ (acting by
pull-back on $\mathcal{M}(B')$).

Formally the base change of $X$ from $B$ is defined as the fibred pro\-duct:


$$\xymatrix{ X'=X\times_B B'  \ar[rr] & & X \ar[dd]\ar[dd]^{\pi}  \\  & &    \\  X'\ar[uu]^{s}\ar[rr] &  & B }$$

The advantage now is that the fibration $X'\rightarrow B'$ admits
a section $s$ given by the identity in $B'$, and therefore there is a
relative Poincare sheaf $\mathcal{E}$ as in the previous case, that restricted on smooth fibers
$X^{'}_{b}\times \widehat{X}'_{b}$, where $\widehat{X}'_{b}$ is
the corresponding dual abelian variety, is just the Poincare
bundle. We are taking as dual fibration of $X'/B'$, the relative
moduli space. If $\mathcal{J}$ is the relative moduli functor, $\mathcal{J} : VAR \rightarrow SETS $,
of semistable sheaves of the fibers containing line bundles of
degree 0 on smooth fibers, over the
smooth locus,
$\mathcal{J}$ is represented by the relative Jacobian
$Pic^{0}(X'/B')$, which is the dual fibration
$\widehat{X'}/B'$. 


Due to the presence of singular fibers, the corresponding
coarse moduli space is not a fine moduli space, but the stack
maybe a FM partner, using the relative Poincar\'e sheaf
$\mathcal{E}$ as kernel of the FM transform, that is,
the relative Poincare sheaf $\mathcal{E}$ is the family representing
an element of $\mathcal{J}(Pic^{0}(X'/B'))$ such that for each
variety $S$ and each $\mathcal{F}\in \mathcal{J}(S)$ there exists
a unique morphism $f:S\rightarrow \widehat{X'}$ satisfying that
$\mathcal{F}\cong f^{*}\mathcal{E}$. Therefore $\mathcal{E}$
induces a natural transformation $\Phi: \mathcal{J}\rightarrow
Hom(-,\widehat{X'}/B')$ 
giving a stack structure $((\widehat{X'}/B'),\Phi)$. It is
universal in the sense that for every other variety $N$ and every
natural transformation $$\psi: Hom(-,N) \rightarrow Hom(-,\widehat{X'}), $$
the following diagram commutes:

$$ \begin{picture}(100,60)
\put(0,50){$\mathcal{J}\rightarrow Hom(-,N)$}
\put(10,40){\vector(1,-1){29}}
\put(50,10){\vector(0,2){35}}\put(45,0){$Hom(-,\widehat{X'})$}
\end{picture} $$



It follows by \cite{HLS} that there is an equivalence of
categories
$$D^{b}(X'/B')\cong D^{b}(\widehat{X}'/B'),$$ where
$\widehat{\rho}: \widehat{X}'\rightarrow B'$ is the dual abelian
fibration. Now the Galois group $G$ acts on bundles on the fibres,
and there is an equivalence between the respective invariant
subcategories $(D^{b}(X'/B'))^{G}\cong
(D^{b}(\widehat{X}'/B'))^{G}$, therefore by the fundamental
theorem of Galois theory, there is an equivalence between the
original categories $D^{b}(X/B)\cong D^{b}(\widehat{X}/B)$.

Now, if the integral functor $\phi^{X\rightarrow \widehat{X}}_{\mathcal{E}}: D^{b}(X)\rightarrow D^{b}(\widehat{X})$ is an equivalence of derived categories, where
$\widehat{\rho}: \widehat{X}\rightarrow B$ is the dual abelian
fibration, from Prop. 2.15 of  \cite{HLS} it follows that there is fibrewise equivalence
$\phi_{b}:D^{b}(X_{b})\rightarrow D^{b}(\widehat{X}_{b})$.

 \cqd

\subsubsection{The twisted fibration}
More generally,
let us consider  now the twisted fibration obtained by considering for each $b\in B$  an automorphism $f\in Aut\,(X/B)$ of the fibration, that is, a biholomorphic map of the abelian variety that defines the fiber,  preserving the polarization class.

Then we get a derived equivalent abelian variety $Y_{b}$ (not necessarily the dual one).

\begin{lemma} The fibre $Y_{b}$ image of $X_{b}:=p^{-1}(b)$ under automorphism $f\in Aut\,(X/B)$ is derived equivalent to $X_{b}$.
\end{lemma}

\begin{proof} If $\mathcal{L}$ is the Poincar\'e bundle over the product $X_{b}\times \hat X_{b}$ of the fibre $X_{b}$ with its dual one as a complex torus, then the Fourier-Mukai transform \begin{eqnarray} \phi^{\mathcal{L}}: D(X_{b})\rightarrow D(Y_{b})  \\ \mathcal{F}\mapsto\mathbf{ f}_{*}(\mathcal{F}\otimes \mathcal{L)}, \end{eqnarray} where $\mathbf{ f}_{*}$ is the pushforward of $f$, gives  the desired equivalence of categories.
\end{proof}

By continuity, for every $t \in B$ there is an open neighborhood $t \in U\subset B$ where for all $b\in U$, $f$ defines an isomorphism $f_{U}:  X_{U} \times \widehat {X_{U}} \simeq Y_{U} \times \widehat{Y_{U}}$.  If $g$ is another isomorphism on another open set $V\subset B$, there is a compatibility condition $f_{U}\circ g^{-1}_{V}=g_{V}\circ f^{-1}_{U}$ on the intersection $U\cap V$.
The fibration $\tilde{\pi}: Y \rightarrow B$ we get in this way,  is a twisted fibration of the original fibration $\pi: X\rightarrow B$.

\begin{lemma}
Given an abelian variety $A$ whose underlying complex tori is defined by a lattice of rank $d$ , the  group of automorphisms  of the lattice is  the linear group $PSL(d, \mathbb{Z})$.
\end{lemma}

{\it Proof.} It follows easily that if we apply a linear map $T\in PSL(d, \mathbb{Z}) $ to the lattice that defines the abelian variety $A$, the corresponding abelian variety is isomorphic to the original one, and therefore $T$ is an automorphism of the lattice. Reciprocally, if two abelian varieties are isomorphic, the corresponding lattices are related by linear transformation up to a scalar, and therefore by an element in $PSL(d, \mathbb{Z}) $. \cqd

The abelian variety $A$ has the structure of an algebraic group finitely generated as an abelian group. Let $e $ be the identity point of this group. It can be checked that any endomorphism
of $A$ that sends the point $e$ to itself is an endomorphism of the algebraic group. Such
endomorphisms form a ring which contains $\mathbb{Z} $ as a subring and for a ÒgenericÓ abelian
variety coincides with it. However the ring of $e$-preserving endomorphisms of $A$ can
be bigger than $\mathbb{Z}$, for example in the case of complex multiplication. Then the ring of automorphisms of the lattice tensored by the rational numbers contains a field of degree $2d$ which is a quadratic extension over a totally real field. Then the group of automorphisms coincides with the units of the ring of integers.

For example if the abelian variety is $A=E^{n}_{\tau}$ with $E_{\tau}$ an elliptic curve with complex multiplication, that is, $\tau$ is a root of a
quadratic polynomial with integral coefficients.

Next, we will construct a gerbe over the twisted fibration $X_{f}$. As before, consider another isomorphism $f_{V}: X_{V}\simeq X_{V}$ of the fibration restricted to another open set $V\subset B$.
There is a compatibility condition $f_{U}\circ f_{V}^{-1}=f_{V}\circ f^{-1}_{U}$ on the intersection $U\cap V$. Let us call $g_{UV}:=f_{U}\circ f^{-1}_{V}$, then $g_{UV}: U\cap V\rightarrow U$ satisfies $g_{UV}=g^{-1}_{UV}$. This data together with the cocycle condition $g_{UV}\circ g_{VW}\circ g_{WU}=1$ on $U\cap V\cap W$, where $f_{W}: X_{W}\simeq X_{W}$ is another isomorphism, defines a gerbe over $X_{f}$ in the sense of Hitchin \cite{Hit}.





\begin{prop}\label{prop1}\label{twistfibration} The derived categories of the two fibrations
$\pi: X\rightarrow B$ and $\tilde{\pi}: Y \rightarrow B$
are fiberwise equivalent if and only if the derived categories of
the two fibrations are equivalent.
\end{prop}
{\bf Proof.}
If the fibrations are derived equivalent from the definition of $Y/B$,  it follows that there is fiberwise equivalence.

Now consider for every $t \in B-\Sigma(p) $, the product  $X_ {t} \times Y_{t}$ of the corresponding abelian variety with its derived equivalent given by an automorphism $f$. Consider a sheaf  $\mathcal{L}\in D^{b}(X\times Y)$ such that for every $\beta \in Y_{t}$, $\mathcal{L}|_{\beta}$ on $X_{t}\times \{\beta\}$ is an element of $Pic^{0}(X_{t})$, that is, if $\mathcal{J}$ is the relative moduli functor, $\mathcal{L}\in \mathcal{J}(Pic^{0}(X/B))$.

The family $\{\mathcal{L}_{t}: t\in B\}$ over the non singular locus satisfies a Calabi-Yau condition, since the generic fiber has trivial canonical bundle, and since the total fibration is Calabi-Yau,  the extended object $\overline{\mathcal{L}}$ over the non singular locus is invertible and determines a Fourier-Mukai transform which depends on $f$. In particular we can take as dual of $X_{U}$ the same fibration, in this case $f\in Aut(X/B)$ and $\mathcal{L}=\Gamma_{f*}\mathcal{L}$  where $\Gamma_{f}$ is the graph of the isomorphism $f$. Then the functor
\begin{eqnarray}
\phi^{\mathcal{L}}: D(X) \to D(Y) \nonumber \\
F \mapsto f_{*}(F  \otimes \pi_{1*}( \bar{\mathcal{L}}))
\end{eqnarray}
defines a FM transform.
 \cqd

\begin{remark} Observe that the case of the dual fibration corresponds to take the identity as $f$, and in this case $\mathcal{P}$ is the Poincar\'e sheaf and satisfies a universal property. For each variety $S$ and each $\mathcal{L}\in \mathcal{J}(S)$ there exists a unique morphism $f:S\rightarrow Pic^{0}(X/B)$ satisfying $\mathcal{L}\cong f^{*}(\mathcal{P})$.
\end{remark}


 As a consequence of the theorem \ref{mainth} and proposition \ref{twistfibration}, given
two Calabi-Yau's 3-folds $X, Y$ fibered by abelian surfaces over
the same base $S$, such that the derived categories of the fibres
are equivalent,
$$D^{b}(X_{s})\cong D^{b}(Y_{s}) \ \ \forall s\in S, \ \
D^{b}(X)\cong D^{b}(Y).$$

A more general problem, is that of Calabi-Yau's fibered by complex tori (not necessarily algebraic),
and this is also interesting from the point of view of mirror
symmetry, and it is according to SYZ mirror prediction.


R. Donagi and  T. Pantev in \cite{DP} start with an elliptic
fibration (admitting a section) and this has many associated genus
one fibrations  without sections codified by the Tate-Shafarevich group (twisted versions of the original fibration coming from replacing
the derived category of sheaves on a space with the derived category of sheaves on a gerbe over the space).
The twist of a fibration $X/B$ is given by
an automorphism $f \in Aut(X/B)$, and the respective relative jacobians are isomorphic $P ic^{0}(X/B) \cong Pic^{0}(X_{f} /B)$.
Let $X$ be an elliptic fibration and let $\alpha$ be a class in its Tate-Shafarevich group. Under rather strong assumptions on $X$ (including smoothness and integrality of the fibers) they are able to understand the relation between the Tate-Shafarevich group of $X$ and the Brauer group of an elliptic fibration $X_{\alpha}$ obtained from $X$ by twisting by the class $\alpha$. Given another class $\beta$ in the Tate-Shafarevich group (compatible with $\alpha$ in the sense specified in \cite{DP}), it is possible to construct a gerbe over $X_{\alpha}$ corresponding to the image class of $\beta$ in $Br(X_{\alpha})$, denoted by $_{\beta}X_{\alpha}$. Interchanging the role of $\alpha$ and $\beta$ one gets another gerbe $_{\alpha}X_{\beta}$ over a fibration $X_{\beta}$ locally isomorphic to $X_{\alpha}$. They conjecture that for $\pi: X\rightarrow B$ an elliptic fibration with $X$ and $B$ smooth, $\alpha$ and $\beta$ a pair of compatible classes ( in the sense of \cite{DP}) there is an equivalence of derived categories of coherent sheaves of weights 1 and -1 $$D^{b}(_{\beta}X_{\alpha}, 1)\stackrel{\sim}{\rightarrow} D^{b}(_{\beta}X_{\alpha},-1).$$

However, in dimension greater than 2 they can only prove the conjecture for smooth fibrations (smooth fibers).

\begin{remark} Due to Theorem 5.11 of \cite{Bas}, if $\mathfrak{X}$ is
a gerbe on a twisted version of $X$, and $\mathfrak{Y}$ the dual
gerbe to $\mathfrak{X}$, then there is an equivalence of
categories,
$$
\cdbc{\mathfrak Y, -1} \cong \cdbc{\mathfrak X, -1}\,.
$$


\end{remark}

\subsubsection*{SYZ mirror conjecture} (A. Stromminge, S.T. Yau, E.
Zaslov). Mirror dual Calabi-Yau manifolds should be fibered over
the same base in such a way that generic fibers are dual tori and
each fiber of any of these two fibrations is a Lagrangian
submanifold.  In particular, each of these fibrations admits a canonical
section that is an $n-$cycle having intersection number 1 with the fiber cycle.

We can take as a base, the moduli space of flat
$SU(n)$-connections $$S=Hom(\pi_{1}(\Sigma),SU(n))/SU(n),$$
as in
the geometric Langland program, and then, the Lagland dual of a
torus $F_{p}$ over a connection $v\in S$ correspond to the dual
torus $F_{p}^{\vee}$ sitting on the dual connection, 
that is, the mirror brane of a torus is the dual torus, and therefore mirror
symmetry in this case is T-duality.

\section{$K3$ fibrations.}
Now we study fibrations $\pi: X\rightarrow S$ that are fibered by polarized $K3$ surfaces. It is of interest from the point of view of Mirror Symmetry, the case in which the total fibration is Calabi-Yau, because
a Calabi-Yau space has two kind of moduli spaces, the moduli space
of inequivalent complex structures and the moduli space of
symplectic structures. Mirror Symmetry should consist in the
identification of the moduli space of complex structures on an
$n-$dimensional Calabi-Yau manifold $X$ with the moduli space of
complexified K\"ahler structures on the mirror manifold
$\widehat{X}$. There is an isomorphism between tangent to moduli
space of complex structures on $X$ and deformations of moduli of
K\"ahler structures on $X'$. 

\begin{defi} A $K3$ surface is a compact complex projective 2-dimensional smooth
variety with trivial canonical bundle and such that its first Betti
number $b_ {1}=0$ vanishes.
\end{defi}

The intersection pairing (which in the complex case coincides with the cup product),
$$< , >:\ H^{2}(X,\mathbb{Z})\times H^{2}(X,\mathbb{Z})\rightarrow \mathbb{Z}, \ \ even\ for \ all \
\alpha\ \in H^{2}(X,\mathbb{Z}),$$
 endows the cohomology group $H^{2}(X,\mathbb{Z})$ with the structure of a lattice which is even, integral, non-degenerate, unimodular and torsion free of rank 22 isomorphic to $H^{2}(X,\mathbb{Z})\cong E_{8}^{2}\oplus U(1)^{\oplus 3}$. Inside the $ H^{2}(X,\mathbb{Z})$ lattice there are two natural sublattices, the N\'eron-Severi sublattice of $X$,  $NS(X)$, consisting of divisors up to algebraic equivalence, (its rank is called the Picard number), and its orthogonal complement, the transcendental lattice $T_{X}=NS(X)^{\bot}$.
 

We have $H^{2}(X, \mathbb{Q}/\mathbb{Z})\cong H^{2}(X,\mathbb{Z})\otimes\, \mathbb{Q}/\mathbb{Z}$, and hece its cohomological Brauer group $$Br(X)\cong (H^{2}(X,\mathbb{Z})/(NS(X))\otimes \mathbb{Q}/\mathbb{Z}.$$

Let $\alpha\in Br(X)$ be represented by a C\u{e}ch 2-cocycle, given along a fixed open cover
$\{U_{i}\}_{i\in I}$ by sections $\alpha_{ijk}\in \Gamma(U_{i}\cap U_{j}\cap U_{k}, \mathcal{O}^{*}_{X})$. An $\alpha-$twisted sheaf $\mathcal{F}$ (along the fixed cover) consists of a pair

$(\{\mathcal{F}_{i}\}_{i\in I}, \{\varphi_{ij}\}_{i,j \in I}),$ where $\mathcal{F}_{i}$ is a sheaf on $U_{i}$ for all $i \in I$ and $$\varphi_{i}: \mathcal{F}_{j}|_{U_{i}\cap U_{j}}\rightarrow \mathcal{F}_{i}|_{U_{i}\cap U_{j}}$$ is an isomorphism for all $i,j \in I,$ subject to the conditions:
\begin{enumerate}
\item $\varphi_{ii}=id;$
\item $\varphi_{ij}=\varphi_{ji}^{-1};$
\item $\varphi_{ij}\circ \varphi_{jk}\varphi_{ki}=\alpha_{ijk}\cdot id.$
\end{enumerate}

Let $e\in H^{1,1}(X,\mathbf{R})\cap H^ {2}(X,\mathbf{Z})$ be the class of an ample divisor. Then $(X,e)$ is a polarized $K3$ surface.
The degree of the polarization is an integer $2d$, such that the scalar product $<e,e>=2d=2rs$ where $d, r, s$
are any positive integers and their greatest common divisor $(r,s)$ is 1.

Consider the fine moduli space $\mathcal{M}(r,e,s)$ parametrizing $e-$stable 
sheaves $E$ on $X$ such that 
$c_{0}(E)=rk(E)=r$, $c_ {1}(E)=e$ and $\chi(E)=r+s$. Here stability means Gieseker stability as considered in \cite{Sim}.
 
The vector $v=(r,l,s)\in \tilde{H}(X,\mathbf{Z})=H^{0}(X,\mathbb{Z})\oplus H^{1,1}(X,\mathbb{Z})\oplus H^{4}(X,\mathbb{Z})$ is a class in the topological $K-$theory $\mathcal{K}_{top}(X)$ of the surface and it is called Mukai vector, it is expressed in terms of a basis as $v=v_{0}+v_{2}+v_{4}$ to distinguish cohomological degrees in $v$. Given another class $w\in K_{top}(X)$, using the product in $K-$theory, we define the Mukai pairing on cohomology as
$$<v,w>=\int_{X}v_{2}w_{2}-v_{0}w_{4}-v_{4}w_{0}.$$

Due to a result of Mukai (see \cite{Muk}), if
$\hat{X}:=\mathcal{M}(r,e,s)$ is non-empty, it is a symplectic
manifold whose dimension is expressed in terms of the Mukai
self-pairing of $v=(r,e,s)$, as $dim\, \mathcal{M}(r,e,s)=<v,v>+2$.
In particular, choosing coordinates $r=2,\ l=0,\ s=4$ for the Mukai
vector, we obtain a 2-dimensional symplectic manifold, that is,  a
$K3$ surface as well, and it is considered the dual mirror surface
to $X$. We will denote it as $\widehat{X}$.

\begin{remark} Observe that we need to assume that the sheaves $E$ parametrized by the moduli space are stable sheaves in order to ensure the smoothness condition. When there do exist estrictly semistable sheaves the moduli space can be singular.
O'Kieran Grady in \cite{Kie} 
studies the moduli space of sheaves $E$ with $c_{2}(E)\geq 4$ and $c_{1}=0$ which is singular exactly along the locus parametrizing strictly semistable sheaves and the smooth locus is symplectic. The author constructs a symplectic desingularization by using Kirwan's method of blowing up loci parametrizing strictly semistable sheaves. The main observation is that the moduli space can be realized as the G.I.T quotient of a Quot scheme acted on by $PGL(N)$.
\end{remark}

\begin{defi} Two K3 surfaces $X, Y$ are said to be FM partners, if there is an equivalence  $D(X)\cong D(Y)$ of their bounded derived categories of coherent sheaves. The set of isomorphism classes of FM partners of $X$ is denoted by $FM(X)$.
\end{defi}

Observe that tensorization by a line bundle $D\in N\,(X)$, 

\noindent $T_{D}:\, \widetilde{H}\,(X)\rightarrow \widetilde{H}(X)$ defined by 

$T_{D}(r,c_{1}(X),s)=(r, c_{1}(X)+r\,D, s+r\,\frac{D^{2}}{2}+D\cdot c_{1})$ gives a birational isomorphism between $\mathcal{M}(r,c_{1}(X),s)$ and the moduli space of stable sheaves $\mathcal{M}'(r,c_{1}(X)+rD, s+r\,\frac{D^{2}}{2}+D\cdot c_{1})$ with Mukai vector $(r,c_{1}+rD,s+r\,\frac{D^{2}}{2}+D\,c_{1}(X))$.

In the same way, reflection $\delta: \widetilde{H}\,(X, \mathbb{Z})\rightarrow \widetilde{H}(X,\mathbb{Z})$ defined by $\delta\,(r,c_{1}(X),s)=(s,c_{1}(X),r)=w$, defines another birational $K3$ surface, that is, the moduli space of stable sheaves $$\mathcal{M}_{X}(r, c_{1}(X),s)$$ and $\mathcal{M}_{X}(w)$ correspond to a birational $K3$ surface, in particular they are derived equivalent.

Next theorem  due to Mukai and Orlov describes given a $K3$ surface,  the Fourier-Mukai partners in terms of the moduli of stable sheaves.

\begin{theorem}{(Derived Torelli)}\label{MO}
The following conditions are equivalent:
\enumerate
\item Mukai's duality is an involution $\hat{\hat{X}}=X$.
\item There exists an equivalence $D^{b}(X)\cong D^{b}(\hat{X})$ induced by the universal family on $X\times \hat{X}$.
\item $\hat{X}$ has a polarization  $e'$, such that $<e',e'>=2d$.
\item There exists a Hodge isometry $f: \tilde{H}(X,\mathbf{Z}) \rightarrow  \tilde{H}(\hat{X},\mathbf{Z})$. 

\end{theorem}


Let us now assume that the base is of dimension 1, even more let $B$ be an algebraic curve $C$ of genus $g$ and let $\pi:Y\rightarrow C$ be a three dimensional projective non-singular variety 
such that for each $t\in C$, $\pi^{-1}(t)=X_ {t}$ is a $K3$ surface. We must have at least 3 singular fibers. From Hironaka's theorem on the resolution of singularities, we may assume that the singular fibers are normal crossing divisors. 

Suppose that on $Y$ we have a polarization class $H$ such that
restricted to the fiber $H|_{X_ {t}}=e$. Let $m$ be the number of
points on $C$ for which the local monodromy operator is of infinite
order. Then the number of singular fibers of  $\pi$ is less or equal
to $2g-2+m$.

Since every $K3$ surface is K\"ahler by Theorem 2 of \cite{To}, the Mumford semi-stable
reduction theorem applies in this case. So we may assume that the
fibers of the map $\pi$ are given locally by $z^{k_ {1}}_ {1} \cdots
z^{k_ {n+1}}=t$,
where $k_ {j}$ is either 0 or 1. 

If $X$ is an elliptic fibration with integral fibers, then the relative Picard functor is
representable and the compactified relative Jacobian $Y$ is derived equivalent to the original one.

\subsection{The dual fibration}
Now, we want to construct, given a $K3$ fibration $p:X\rightarrow C$ with a fixed polarization, its dual fibration, that should play the role of mirror fibration.
We can assume that the fibration is Calabi-Yau. Since  the singular fibers are normal crossing divisors,  and the total space and the base are projective varieties the fibration morphism is automatically proper. We are assuming that the fibers are equidimensional and therefore the morphism is flat.  By the theorem of U. Person and H. Pinkham \cite{PP}, there exists a birational map $\varphi: X\rightarrow X'$  where $X'$ has trivial canonical bundle and it is an isomorphism over the smooth locus such that the following diagram is commutative:

 $$ \xymatrix{ X  \ar[rr]^{\varphi}  \ar[rd]^{\pi} &  & X'\ar[ld] _{\pi '}    \\
&  B & } $$

Now, by Bridgeland theorem (see \cite{Bri}), two birational  3-folds have equivalent derived categories.

The idea is to replace each fiber by its derived categorically equivalent one. The problem is that singular fibers can appear for such a fibration. A natural idea is as in the case of abelian fibrations, to contruct the dual fibration away from the singularities, and then trying to extend it over the singular locus.

\subsubsection{Relative moduli spaces}
The moduli space $\mathcal{M}^{l}(X)$ of semistable sheaves on $X$
with respect to a fixed polarization $l$, in general has infinitely
many components, each of which is a quasi-projective scheme which
may be compactified by adding equivalence classes of semistable
sheaves. An irreducible component $Y\subset \mathcal{M}^{l}(X/C)$ is
said to be fine if $Y$ is projective and there exists a universal
family of stable sheaves, that is, and object of $D^{b}(X\times Y)$
inducing a derived equivalence.

In general, $Y$ has not to be fine, but for any moduli problem of
semistable sheaves on a space $X$, we can find a twisted universal
sheaf on $X\times Y^{s}$, where $Y^{s}$, denotes the stable part of
$Y$, (see \cite{Cal}). The twisting depends only on the moduli
problem under consideration, and therefore we can view it as the
obstruction to the existence of a universal sheaf on $X\times
Y^{s}$.

\begin{theorem}\label{theo1}
 Given a non singular fibration $p:X\rightarrow C$ by $K3$ surfaces 
 with a polarization class $H$ of degree $d$,
 there exists at least a dual fibration which is derived equivalent to the original one and corresponds to a connected component of the relative
 moduli space $\mathcal{M}^{l}(X/C)$.  
\end{theorem}
\begin{proof}[Proof.]
Let $\Sigma(p) \hookrightarrow C$ be the discriminant locus of $p$,
that is, the closed subvariety in  $C$
corresponding to the singular fibers.

 For every $t \in C-\Sigma(p) $, consider the $K3$ surface $X_{t}$ and its corresponding Mukai vector $(r_{t},e, s_{t})$, where  $2r_{t}s_{t}=(H_{t})^{2}=H^{2}=2d$.
 By Theorem \ref{MO}, we associate to $X_{t}$ a 2-dimensional moduli space $\mathcal{M}(r_{t},e,s_{t})$ which is
 a FM partner, that is, it has the same derived category. 
 We observe that although the degree of the polarization is constant in $t$, the rank of the fibers can jump for some $t \in C$. However the condition of the Picard rank being one is open in the Zariski topology and it determines an open set
 $$C^{1}:=\{t\in C|\,\, NS(X_{t})=\mathbb{Z}H_{t}\}.$$
 
 Now if $s\in C^{1}$, then $H|_{X_{t}}=H_{t}=l$ is an ample divisor and since the number of Mukai partners depends on the prime decomposition
 $l^{2}=2d=2p_{1}^{e_{1}}\ldots p^{e_{m}}_{m}$, where $k\geq 0$, $e_{i}\geq 1$ and $p_{i}$ primes with $p_{i}\neq p_{j}$, if $i\neq j$,  there is a description of the FM partners of the $K3$ surface  in terms of the Mukai vectors of the moduli spaces associated (see \cite{St}). We need to single out a unique Mukai dual $K3$ surface. For example, the reflected Mukai vectors $(r_{t},e, s_{t})$ and $(s_{t},e,r_{t})$ give isomorphic moduli spaces $\mathcal{M}(r_{t},e,s_{t})\cong \mathcal{M}(s_{t},e,r_{t})$ even if the original $K3$ surfaces are not isomorphic. Thus, this choice gives rise to  different dual fibrations.
 Given a disjoint partition $I=\{j_{1},\ldots,j_{s}\}$ and $J=\{j_{s+1}\ldots, j_{m}\}$ of $\{1,\ldots, m\}$. We consider the corresponding Mukai vector $v^{I}_{J}=(p_{j_{1}}^{e_{j_{1}}}\cdot\ldots\cdot p^{e_{j_{s}}}_{j_{s}},h,p^{e_{j_{s+1}}}_{j_{s+1}}\cdot \ldots \cdot p^{e_{j_{m}}}_{j_{m}})$. If the Mukai vectors $v^{I_{1}}_{J_{1}}=(r_{1},l,s_{1})$ and $v^{I_{2}}_{J_{2}}=(r_{2},l,s_{2})$ associated to two different partitions coincide $v^{I_{1}}_{J_{1}}=v^{I_{2}}_{J_{2}} $, they give isomorphic moduli spaces $\mathcal{M}(v^{I_{1}}_{J_{1}})\cong \mathcal{M}(v^{I_{2}}_{J_{2}})$ of stable sheaves on $X_{t}$. 
 
 If the rank of the Neron Severi group $NS(X_{t})$ is bigger than 12, according to Morrison \cite{Mo},  there exists a torsion free semistable bundle on $X_{t}$, and the choice of dual $K3$ surface is unique in this case.
 
 Consider  the product  $X_ {t} \times \hat{X}_{t}$  of the corresponding $K3$ surface $X_{t}$ with its Mukai dual.


 Then we consider the universal family $\mathcal{P}_{t}$ over the product $X_{t}\times \hat{X}_{t}$. 
Proceeding as in Proposition 2.9 of \cite{MT}, extending the family
$\mathcal{P}:=\{\mathcal{P}_{t}: t \in B\}$ over the non singular
locus by Deligne theorem (\cite{Del}), the class of the polarization
is invariant by the action of the monodromy group of the singular
fibres, thus $\mathcal{P}$ extends to an object $\mathcal{F}$ over
the whole fibration.

The family does not need to be universal, but according to
C\v{a}ld\v{a}raru (see \cite{Cal}), a quasi-universal or twisted
universal family sheaf always exists and thus the dual fibration
$(X/C)^{\vee}$  is the coarse moduli space induced by $\mathcal{F}$.
The fibration constructed thus far,  is a connected component $M$ of
the relative moduli space $\mathcal{M}^{l}(X/C)$ of stable sheaves
on $p$ with respect to the polarization, (Prop. 3.4. of \cite{BM}).
There exists a unique $\alpha$ in the Brauer group $Br(M)$ of $M$
with the property that an $p^{*}_{M}\alpha^{-1}$ twisted universal
sheaf exists on $X\times M$, where $p_{M}$ is the projection map
from $X\times M$ to M, and it is the obstruction to the existence of
a universal sheaf on $X\times M$. This twisted universal sheaf
yields an equivalence (Theorem 1.2 of \cite{Cal}).
$$D^{b}(M,\alpha)\cong D^{b}(X).$$
So both fibrations are derived equivalent.

\end{proof}




\begin{coro}\label{finecomp}
There exists at least one  fine component of the relative moduli space or equivalently a sheaf $\mathcal{E}$
on a non singular fiber with fixed Mukai vector. 
\end{coro}
A closed point of a relative moduli space corresponds to a sheaf
$\mathcal{E}$ on a fibre (not to a sheaf on the whole fibration).
Let $X_{s}$ be a $K3$ surface or an abelian surface. The tangent
space at that point to the moduli space of sheaves $M(X/S)$ on the
fibration, can be identified with
$$T_{M}(\mathcal{E})\cong Ext^{1}_{S}(\mathcal{E},\mathcal{E}).$$

If $Ext^{2}_{S}(\mathcal{E},\mathcal{E})=0$, then $M$ is smooth at
$\mathcal{E}$. There are bounds (Corollary 4.5.2 of
\cite{HL}), $$ext^{1}(\mathcal{E},\mathcal{E})\geq dim_{[\mathcal{E}]}M \geq
ext^{1}(\mathcal{E},\mathcal{E})-ext^{2}(\mathcal{E},\mathcal{E}).$$ 

In general to construct such components $Y$ of the relative moduli space,
we assume that there exists a divisor $L$ on $X$ and integer numbers $r, s>0$,
such that there exists a sheaf $\mathcal{E}$ on a non singular fiber $X_ {t}$
which is stable with respect to $H_ {t}$ and $s=ch_ {2}(\mathcal{E})+r$.
The component $Y(\mathcal{E})$ containing the class of the sheaf $\mathcal{E}$ is a fine projective moduli space
and the fibration $q:Y\rightarrow B$ is equidimensional. 
Thus there is a universal family on the product $Y \times
Y(\mathcal{E})$ that gives the equivalence of the derived categories
of both fibrations over $B$.



\subsubsection{Derived categorical equivalences of the $K3$ fibration and mirror symmetry}
Now we are going to proceed in the inverse way, by considering  components $Y$ of the relative moduli space
$\mathcal{M}^{e}(X/B)$ of stable
sheaves of fixed numerical invariants $r, s$ (hence fixed Hilbert polynomial)
on the fibers of $p$ which are stable with respect to a fixed polarization $e$.  
In order to avoid parametrizations,
we can look at the component of the relative moduli space of sheaves containing line bundles in degree 0.

\begin{prop} \label{prop1} Every fine projective component $Y$
of the relative moduli space $\mathcal{M}^{e}(X/B)$ of stable
sheaves with respect to a fixed polarization $e$ is derived
equivalent to the original Calabi-Yau fibration $(X/B)$ and
therefore are derived equivalent between them. Conversely, any
projective variety derived equivalent to the original fibration is a
component of the relative moduli space.
\end{prop}
\begin{proof}[Proof.]
By Corollary \ref{finecomp} we can consider  components $Y$ of the relative moduli space $\mathcal{M}^{e}(X/B)$ of stable
sheaves 
on the fibers of  the CY fibration $(X/B)$,  stable with respect to
the polarization $e$.  It is a fine moduli space, so there is a
universal sheaf $\mathcal{P}$ over the product $X\times Y$.
Bridgeland and Maciocia proved in \cite{BM} that $Y$ is a
non-singular projective variety, $\widehat{p}:Y\rightarrow B$ is a
$K3$ fibration and the integral functor $D^{b}(Y)\rightarrow
D^{b}(X)$ with kernel $\mathcal{P}$  is an equivalence of derived
categories, that is, a Fourier-Mukai transform. It is Calabi-Yau
because one has $D^{b}(X)\cong D^{b}(Y)$.

Now, we start with an equivalence $D^{b}(Y)\cong D^{b}(X)$, then by
a result of Orlov \cite{Or1}, it is given by an object
$\mathcal{E}\in D^{b}(X\times_{B}Y)$ which satisfies a Calabi-Yau
condition and thus by Theorem \ref{theo1} this defines a fine
component of the relative moduli space. All the equivalences of the
original fibration are obtained in this way.
\end{proof}

Now the next question is what is the connection with the Mirror Symmetry program, that is, the component $Y$
can be interpreted as a mirror to the original fibration.  
This is related with the problem of all CY manifolds giving rise to the same derived category and for a
description of all the equivalences of it. 
Let $\mathcal{C}$ be the set of all sheaves on non singular fibers
giving rise to fine projective components of the relative moduli
space. By Proposition \ref{prop1} and Corollary \ref{finecomp},
$\mathcal{C}$ is in bijection with the set of derived equivalences
of the original fibration $X/B$.

Homological mirror symmetry 
states that there should be an equivalence of categories behind mirror duality, one category being the derived category of coherent sheaves on a Calabi-Yau manifold $X$ and the other one being the Fukaya category of the mirror manifold $X $.
As a consequence, two Calabi-Yau manifolds that have the same mirror have equivalent derived categories of coherent sheaves. 

For two sheaves $\mathcal{E}$ and $\mathcal{F}$ of  $\mathcal{C}$,
the corresponding components $Y(\mathcal{E})$ and $Y(\mathcal{F})$
of the relative moduli, have the same mirror, that is, the original
fibration $X/B$, and this is according to homological mirror
symmetry conjecture. 





\subsection{Examples of K3 fibred Calabi-Yau folds}
Finally we give some specific examples of Calabi-Yau fibrations by $K3$ surfaces.
One example of the $K3$ fibred Calabi-Yau threefold $X$ is obtained by resolving singularities of the
degree 8 hypersurface $\widehat{X}\subset \mathbf{P}_{1,1,2,2,2}$.

The K\"ahler cone of X is generated by positive linear combinations of the linear system
$H = 2l + e$ and $l$ where $e$ is an exceptional divisor coming from blowing-up
a curve of singularities and the linear system $l$ is a pencil of K3 surfaces.


\subsection{Elliptically fibered Calabi-Yau 3-fold}
In this section we study Calabi-Yau 3-folds whose fibers are elliptic $K3$ fibrations, that is $K3$ surfaces whose fibers are elliptic curves. Thus,
for elliptic Calabi Yau's, the problem is reduced to study derived
categories of elliptic curves. Any complex torus can be expressed
as the quotient of $\mathbb{C}$ by a lattice of the form
$\bigwedge_{\tau}=\{\mathbb{Z}+\mathbb{Z}\tau \}$ with $\tau\in
\mathbb{H}$ an element of the upper-half plane. We denote the
complex torus given by the lattice $\bigwedge_{\tau}$ by
$E_{\tau}$. The complex torus is called an elliptic curve if it is
an algebraic curve, in other words if it has a distinguished
point.

The Torelli theorem for smooth elliptic curves asserts that two elliptic curves are isomorphic if and only if their bounded derived categories of coherent sheaves are equivalent (see \cite{Be}).


\begin{prop}\label{propMS} Let $X/B$ and $Y/B$ be two elliptic $K3$ fibrations over $B$ such
that for all $ b\in B$, the corresponding Teichm\"uller parameters
$\tau, \tau'$ of $X_{b}$ and $Y_{b}$ are related by a linear
transformation $A \in SL(2,\mathbb{Z})$, then the two fibrations
are derived equivalent.
\end{prop}

{\bf Proof.} Two lattices $\bigwedge_{\tau_{1}}$ and
$\bigwedge_{\tau_{2}}$ determine the same complex torus, when
there exists a biholomorphic map from $E_{\tau_{1}}$ to
$E_{\tau_{2}}$; $\bigwedge_{\tau_{1}}$ and $\bigwedge_{\tau_{2}}$
give holomorphic tori if and only if
$\tau_{1}=A\,\tau_{2}$ for some element  $A=\left(\begin{array}{ll} a &  b \\
c & d \end{array}\right)$ $\in PSL(2,\mathbb{Z})$ with its usual
action on $\mathbb{H}$. Then the corresponding elliptic curves are
isomorphic and by Torelli theorem they are derived equivalent.
\cqd

The second modulus of a Calabi-Yau space is the k\"ahler class
$[w]\in H^{2}(E,\mathbb{C})$ that we can parametrize with $t\in
\mathbb{H}$ as $\int_{E}w=2\pi i \, t$. Mirror symmetry for
elliptic curves is simply the interchange of $\tau$ and $t$.

\subsubsection*{The STU model}
If $X\rightarrow S$ is a $K3$ surface fibred over a curve $S$, the
generic fibre is an elliptic curve and $S$ is a projective line. The
fibered $K3$ surfaces of the STU model are themselves elliptically
fibred. It is a particular non-singular projective $CY-$ 3-fold $X$
equipped with a fibration $X\rightarrow \mathbf{P}^{1}$, $X_
{\xi}=\pi^{-1}(\xi)$. Except for 528 points $\xi \in
\mathbf{P}^{1}$, the fibers are non-singular elliptically fibered
$K3$ surfaces. The 528 singular fibers $X_{\xi}$ have exactly 1
ordinary double point singularity each one. In this particular
example of fibred CY, we know what is the dual fibration by theorem
\ref{theo1} which turns to be also the mirror fibration. For every
$t\in \mathbf{P}^{1}$, $X_ {t}$ admits a fibration $f:X_
{t}\rightarrow B$ by elliptic curves. Then $\forall b  \in B$, $X_
{t,b}$ and $Y_ {t,b}$ are related by linear transformation or mirror
transformation (according to \ref{propMS}).



\end{document}